\documentclass{article}       
\usepackage{graphicx}
\usepackage{amsfonts}
\usepackage{amssymb}
\usepackage{amsmath}
\usepackage{amsthm}
\usepackage{textcomp}
\usepackage{bm}
\usepackage{graphicx,epsfig}

\newtheorem{assumption}{Assumption}[section]

\newtheorem{corollary}{Corollary}[section]
\newtheorem{theorem}{Theorem}[section]
\newtheorem{lemma}{Lemma}[section]

\begin{document}

\title{Convergence analysis of GMRES for the Helmholtz equation via pseudospectrum \thanks{This work was supported by the Alfred Kordelin's Foundation and Academy of Finland projects 13267297 and 14341}}


\author{Antti Hannukainen\footnote{Aalto University, Department of Mathematics and Systems Analysis, P.O. Box 11100, FI-00076 Aalto, Finland, email:antti.hannukainen@aalto.fi}}



\maketitle

\begin{abstract}
Most finite element methods for solving time-harmonic wave - propagation problems lead to a linear system with a non-normal coefficient matrix.  The non-normality is due to boundary conditions and losses. One way to solve these systems is to use a preconditioned iterative method. Detailed mathematical analysis of the convergence properties of these methods is important for developing new and understanding old preconditioners. Due to non-normality, there is currently very little existing literature in this direction. In this paper, we study the convergence of GMRES for such systems by deriving inclusion and exclusion regions for  the pseudospectrum of the coefficient matrix.  All analysis is done a priori by relating the properties of the weak problem to the coefficient matrix. The inclusion is derived from the stability properties of the problem and the exclusion is established via field of values and boundedness of the weak form. The derived tools are applied to estimate the pseudospectrum of time-harmonic Helmholtz equation with first-order absorbing boundary conditions,  with and without a shifted-Laplace preconditioner. 


\end{abstract}

\section{Introduction} Several different strategies for discretizing time-harmonic wave propagation problems using finite elements have been proposed in the literature. For typical problems, most of these strategies  lead to a linear system with a large, sparse, indefinite and non-normal coefficient matrix. The indefiniteness is due to the wave-nature of the problem and the non-normality arises either from losses or truncation of infinite domains to finite ones. The large size of the system is due to the number of degrees of freedom required to resolve an oscillating solution. Because of their properties, the linear systems related to time-harmonic wave propagation problems are difficult to solve. Memory is an issue with direct solvers and lack of efficient preconditioners with iterative ones. 

In order to develop new preconditioners and to understand old ones, it is important to know their effect on the convergence properties of the applied iterative method. Unfortunately, the convergence of iterative methods for linear systems with a non-normal coefficient matrix is a difficult subject of study. When the non-normality is significant, the iterative properties can be very different from the ones indicated by eigenvalues, \cite{Sa:2003,Gr:1997}. Similar difficulties are met with other properties related to the non-normal matrices, e.g., behavior of matrix exponentials cannot be predicted by eigenvalues, \cite{TrEm:2005}. Determining when the non-normality has a significant impact to iterative properties is complicated. The simplest way to estimate the impact is to compute one of the commonly used scalar measures of non-normality, e.g., $\| AA^* - A^*A \| \| A\|^{-1}$, the conditioning of eigenvectors or the conditioning of individual eigenvalues, \cite{Tr:99}. However, except for the first one, these measures are not computable for large matrices. In addition, they can vary considerably even for relatively small systems  \cite{Tr:99}.

The convergence of preconditioned iterative methods has been extensively studied in the context of finite element methods, \cite{ToWi:05}. However, majority of the research has focused on real valued symmetric positive definite problems. The finite element discretization of these problems also leads to symmetric positive definite linear systems, which are solved using the preconditioned conjugate gradient method (PCG). The aim in the analysis of these methods is to estimate the convergence rate before computations.  Only few of the existing works deal with indefinite linear systems, \cite{Ys:1989,Jay:03,CaWi:92,CaWi:1993,BrPaLe:1993}, and even fewer with non-normal indefinite ones, \cite{GiErVu:07,Ha13,GiEr:2006}.

Most preconditioners for finite element discretizations of elliptic weak problems have been analyzed by using the abstract framework of  Schwarz methods,  \cite{ToWi:05}. This framework is based on studying the properties of the underlying weak problem instead of the linear system. The convergence of PCG is related to the weak form via Rayleigh quotients. Such analysis is done in the inner product induced by the bilinear form. As Rayleigh quotients are the first step in the existing analysis, it does not carry over to complex valued,  indefinite, non-normal linear systems. Such systems require a different set of analytics tools. 

There currently exists three different ways to analyze iterative properties of a non-normal matrix \cite{Em:1999}: to study the field of values (FOV), pseudospectrum, or to include conditioning of eigenvectors to the convergence estimates.  For time-harmonic wave-equations, estimating eigenvector conditioning before the matrices are constructed seems to be complicated and thus this approach is not suitable for our purposes. FOV has been applied to analyze the preconditioned time-harmonic Helmholtz equation e.g. in \cite{Ha13}. However, FOV is always a convex set containing all eigenvalues of the matrix. As we will see, the spectrum of the problems we are interested in curls around the origin making FOV based methods unsuitable for our purposes. In contrast, the pseudospectrum can be a non-convex set and as we will show it can be estimated a priori, making it the best option of the three for this work.

In this paper, we study pseudospectrum as a tool for relating the properties of the weak problem to the convergence of the GMRES method. We derive convergence estimates for GMRES by establishing inclusion and exclusion regions for the pseudospectrum. The exclusion region is derived from the stability estimates of the weak problem and the inclusion region is based on then relation between pseudospectrum and FOV. In several cases, an inclusion for FOV can be easily obtained based on continuity properties of the weak form. All analysis is done a priori, so that the regions can be obtained without constructing the actual matrices or performing computations with them. The derived bounds are explicit in the relevant parameters of the problem, e.g.,  mesh size, wave-number and the losses. The presented analysis relies on general properties of the weak problem, stability and continuity so it is possible that it can be applied to other preconditioners and problems.


The paper is organized as follows. We begin with some preliminaries  and proceed to give estimates relating pseudospectrum to a weak problem. After establishing these abstract results, we apply them to three example problems. We begin the examples by considering the Poisson problem, which is included for easy reference on what kind of information the derived estimates can deliver. Then we apply the presented tools to time-harmonic Helmholtz equation with and without a shifted-Laplace preconditioner. We end the paper with a discussion of the presented material.

\section{Preliminaries}

Our model problem is:  Find $u \in V$ such that
\begin{equation}
\label{eq:weak}
a( u,v) = L(v) \quad \forall \; v \in  V,
\end{equation}
where $V$ is some finite element space, $a(\cdot,\cdot) : V \times V\rightarrow \mathbb{C}$ is a sesquilinear form and $L(\cdot) : V \rightarrow \mathbb{C}$ an antilinear functional. The finite element space $V$ is spanned by basis $\left\{ \varphi_i \right\}_{i=1}^N$ so that every function $u \in V$ admits the representation 
\begin{equation*}
u = \sum_{i=1}^{N} (\vec{x}_u)_i \varphi_i,
\end{equation*}
in which the vector of coefficients $\vec{x}_u \in \mathbb{C}^N$. Problem (\ref{eq:weak}) leads to the linear system
\begin{equation*}
A \vec{x}_u = \vec{b},
\end{equation*}
where $A \in \mathbb{C}^{N\times N}, \vec{b}\in \mathbb{C}^N$, $A_{ij} := a( \varphi_j, \varphi_i)$ and $\vec{b}_i := L(\varphi_i)$. Hence, the sesquilinear form and the matrix $A$ are related as 
\begin{equation}
\label{eq:A_2_WEAK}
a(u,v) := \vec{x}_v^* A \vec{x}_u,
\end{equation} 
where $*$ - is the conjugate transpose. The  properties of the matrix $A$ will depend on the properties of the sesquilinear form and the basis functions via the above equation.

We will describe the actual problem and discretization in detail in Section 4. For now, let us note that when the sesquilinear form $a$ is related to the time-harmonic Helmholtz equation with absorbing boundary conditions, the matrix $A$ can be very large. This is due to the facet that the finite element mesh size has to be sufficiently fine before finite element method can produce accurate results, see \cite{BaIh:95,BaIh:97}. Typical engineering rule of thumb is to use ten degrees of freedom per one wave-length. For example, a cube for which each dimension is ten wave-lengths long requires one to use $10^6$ degrees of freedom, this is, $N = 10^6$ or larger.

In the following, we assume that problem (\ref{eq:weak}) has a unique solution and admits some kind of a stability estimate. Stability estimates are typically derived under additional assumptions on the domain and the  antilinear functional $L$. In general, the functional $L$ can be from the space $V^\prime = \{f :V \rightarrow \mathbb{C} \; | \; \bar{f} \in V^* \; \}$, where $V^*$ is the dual space of $V$. As such functionals can be quite badly behaving, stability estimates are often derived under the assumption $L \in W^\prime$, where $V \subset W$. In this spirit, we make the following assumption. 

\begin{assumption} 
\label{ass1}
Let $W$ be a Hilbert space, $V \subset W$, $L \in W^\prime$ and $u \in V$ be the unique solution to problem (\ref{eq:weak}). Then there exists a constant $C_S> 0$ independent of $u$ and $L$ such that 
\begin{equation}
\label{eq:stab}
\| u \|\leq C_S\|L \|_{W^\prime}
\end{equation}
where $\|\cdot \|$ is a norm on $V$ and $\|\cdot \|_{W^\prime} := \sup \{ \; \left| L(w) \right | \; | \; w \in W \mbox{ and } \|w \|_W = 1 \; \}$.
\end{assumption}

The pseudospectrum of a matrix $A \in \mathbb{C}^{N \times N}$, $\Lambda_\epsilon(A)$,  is a family of sets depending on a parameter $\epsilon > 0$. The sets in the family  are defined as 
\begin{equation*}
\Lambda_\epsilon(A) := \left\{ \; z \in \mathbb{C} \; | \; \left| (zI - A)^{-1} \right| \geq \epsilon^{-1} \right\},
\end{equation*}
in which $|\cdot|$ is the standard spectral norm. When the matrix $(zI-A)$ is singular, we define $| (zI-A)^{-1}| = \infty$. The notation $| \cdot |$ is also used for the Euclidian norm of a vector. Clearly, the pseudospectrum can also be characterized as 
\begin{equation*}
\Lambda_\epsilon(A) := \left\{ \; z \in \mathbb{C} \; | \; \sigma_{min} \right( zI - A \left) \leq \epsilon \right\},
\end{equation*}
in which we denote the smallest singular value of a matrix $B \in \mathbb{C}^{N \times N}$ as $\sigma_{min}(B)$.

The pseudospectrum was independently proposed by several authors as an extension of the spectrum, suitable to study the properties non-normal matrices, \cite{TrEm:2005}. The pseudospectrum has been extensively studied in the literature, see e.g. \cite{Tr:99,Tr:97,TrEm:2005}. In the following, we write $\Lambda_\epsilon(A) = \Lambda_\epsilon$, when the matrix $A$ is clear from the context.

In the derivation of the inclusion region, we take advantage on the relation between FOV and pseudospectrum. The FOV of a matrix $A \in \mathbb{C}^{N \times N}$ is defined as the set 
\begin{equation}
\label{eq:FOV}
FOV(A) := \left\{ \frac{ \vec{x}^*A \vec{x}}{\vec{x}^* \vec{x}} \; | \; \vec{x} \in \mathbb{C}^N \mbox{and } \vec{x}\neq 0 \;  \right\}.
\end{equation}
The set $FOV(A)$  is convex, compact and contains all eigenvalues of $A$. As we will see, coarse inclusion for $FOV(A)$ can be obtained by using it's close relation with the sesquilinear form. We postpone stating the relation between pseudospectrum and FOV to Section 3, where we have introduced sufficient notation for proving it. 

Both  pseudospectrum and FOV can be related to convergence of the GMRES method, \cite{Em:1999,Gr:1997,Sa:2003}.  The approximation error for the solution $\vec{x}_i$ generated by GMRES on step $i$ is measured as $|\vec{r}_i |$, where $\vec{r}_i$ is the residual, $\vec{r}_i = A \vec{x}_i - \vec{b}$. There holds that 
\begin{equation}
\label{eq:GMRES_B1}
| \vec{r}_i | = \inf_{\substack{p \in \tilde{P}_i  \\ p(0)=1} } | p(A) \vec{r}_0 |,
\end{equation}
in which $\tilde{P}_i$ is the space of monic polynomials of degree $i$. The matrix valued polynomial in the above minimization problem can be evaluated using Dunford integral \cite{Yo:1995,Em:1999}. Let the open set $U \subset \mathbb{C}$ be such that  $\sigma(A) \subset U$ and $\partial U$ is the union of rectifiable positively oriented Jordan curves. The set $\sigma(A)$ is the spectrum of $A$. Application of the Dunford integral gives 
\begin{equation}
\label{eq:dunford}
p(A) = \frac{1}{ 2 \pi \mathrm{i} } \int_{\partial U} p( z ) ( z I- A )^{-1}d z.
\end{equation}
This integral can be used to derive estimates for equation (\ref{eq:GMRES_B1}). Let $\tilde{\Lambda}_\epsilon$ satisfy the assumptions made on the set $U$ and in addition let 
\begin{equation*}
\Lambda_\epsilon \subset \tilde{\Lambda}_\epsilon.
\end{equation*}
This is, $|zI - A |\leq \epsilon^{-1} \; \forall \; z \in \partial \tilde{\Lambda}_\epsilon$. In our case, $\tilde{\Lambda}_\epsilon$ is an approximation for the pseudospectral set. Estimating the integral gives 
\begin{equation}
\label{eq:GMRES_PA}
| p(A) \vec{r}_0 | \leq |p(A)|| \vec{r}_0 |\leq \frac{ | \vec{r}_0| | \partial \tilde{\Lambda}_\epsilon|}{2 \pi \epsilon } \sup_{z \in \tilde{\Lambda}_\epsilon}|p(z)|.
\end{equation}
Combining equations (\ref{eq:GMRES_B1}) and (\ref{eq:GMRES_PA}) leads to the GMRES convergence estimate 
\begin{equation}
\label{eq:GMRES_B2}
\frac{ | \vec{r}_i |}{| \vec{r}_0| } =\frac{ | \partial \tilde{\Lambda}_\epsilon|}{2 \pi \epsilon } \inf_{ \substack{p \in \tilde{P}_i  \\ p(0)=1}} \sup_{z \in \tilde{\Lambda}_\epsilon}|p(z)|. 
\end{equation}
As we will illustrate in Section 4, this bound is useful for deriving  worst case behavior of the GMRES convergence rate.

The convergence bound (\ref{eq:GMRES_B2}) is meaningful only if one can solve the complex polynomial minimization problem. Typically, the set $\tilde{\Lambda}_\epsilon$ is replaced with a larger set on which the minimization problem can be solved analytically. In simple cases, $\tilde{\Lambda}_\epsilon$ can replaced with a circular or an elliptical domain, \cite{Sa:2003}. Due to the constraint $p(0)=1$, this approach gives useful information only when the circle or ellipsoid containing $\tilde{\Lambda}_\epsilon$ does not contain the origin. When this is the case, one can try to apply so-called bratwurst shaped domains  \cite{Koch:00}. As the name suggests, a bratwurst shaped domain can curl around the origin and it can be used to derive convergence estimates for the minimization problem. The bratwurst shaped domains can be applied with the inclusion and exclusion regions derived in this paper.  However, the construction given in \cite{Koch:00} is not simple, and cannot yield easy to use a priori bounds.

In order to verify the analytically derived inclusion and exclusion results, we compute examples of the pseudospectral sets. Several different strategies for computing pseudospectrum have been proposed, see \cite{Tr:99} and references therein. Several software packages, such as EigTool, are also freely available \footnote{see the Pseudospectral Gateway, http://www.cs.ox.ac.uk/pseudospectra/}. 

To have full control over the computation of the pseudospectrum, we have chosen to use our own implementation of GRID - approach
to compute pseudospectral sets. In the GRID-approach, a mesh is placed in the complex plane and the norm of the resolvent is computed for each grid point. The computed data is used to isolines describing the set.  In the simplest case, the norm is computed as the largest singular value of the matrix $(z-A)^{-1}$. Clearly, such an approach is very expensive for large number of points and large matrices. The process can be sped up by adapting the computational grid to the resolvent norm or by using a suitable matrix factorization to speed up the evaluation of the largest singular value. We have opted to speed up the computation by using an adaptive strategy to refine the computational grid. An initial triangular grid is placed in the complex plane. The grid is iteratively refined to conform to the shape of the resolvent norm. We use a refinement strategy based on splitting triangles intersecting with pre-specified level sets of the resolvent norm. This guarantees higher resolution at interesting regions of the complex plane.

\section{Abstract Framework}

In this section,we derive inclusion and exclusion regions for the pseudospectral set. For this purpose, it is easier to bound the complement of $\Lambda_\epsilon$, i.e.
\begin{equation}
\label{eq:cps}
\Lambda^c_\epsilon := \left\{ \; \left| (zI - A)^{-1} \right| < \epsilon^{-1} \right\}.
\end{equation}
The inclusion and exclusion regions will lead to a set containing the pseudospectrum. If the boundary of this set is a rectifiable Jordan curve, it can be used in connection with equation (\ref{eq:GMRES_B2}) to compute convergence estimates for the GMRES method. The exclusion will be a disc around the origin. For the results to be meaningful, the exclusion should not be fully contained in the inclusion.  If this is the case, the polynomial minimization problem in equation (\ref{eq:GMRES_B2}) does not tend to zero  and the bound does not provide useful information. This has to be studied separately for each problem.

When $(zI-A)$ is non-singular, the matrix norm in equation (\ref{eq:cps}) is defined as
\begin{equation}
\label{eq:norm_def}
| (zI - A)^{-1}| := \sup_{u \in V} \frac{|(zI-A)^{-1}\vec{x}_u |}{|\vec{x}_u|}.
\end{equation}
To eliminate the inverse and to establish a connection to the weak problem, we define an auxiliary vector $\vec{x}_v \in \mathbb{C}^N$ such that 
\begin{equation}
\label{eq:aux_def}
(zI-A) \vec{x}_v = \vec{x}_u.
\end{equation}
Estimates for the resolvent norm are derived using the auxiliary variable. First, we establish the stability bound $|\vec{x}_v|\leq f(z) | \vec{x}_u|$. When $f(z)$ is bounded from above, this implies that $(zI - A)$ is non-singular. In this case, the auxiliary vector is uniquely defined and 
\begin{equation*}
 \vec{x}_v = (zI-A)^{-1}\vec{x}_u.
\end{equation*}
The norm (\ref{eq:norm_def}) can be estimated using the stability estimate for $\vec{x}_v$ as
\begin{equation}
\label{eq:norm_aux}
| (zI - A)^{-1}| = \sup_{u \in V} \frac{|\vec{x}_v |}{|\vec{x}_u|} \leq f(z). 
\end{equation}
 We begin by taking advantage of the stability of the weak problem, Assumption \ref{ass1}. Due to the duality between coefficient vectors and functions,  stability of the weak problem implies stability of the linear system. As all finite dimensional norms are equal, there exists positive constants $\alpha, \alpha_W > 0$ independent of $u$ such that 
\begin{equation}
\label{ass2}
\alpha | \vec{x}_u| \leq \| u \|  \quad \mbox{ and } \quad \alpha_W | \vec{x}_u| \leq  \| u \|_W  \quad \forall u \in V.
\end{equation}
When the derived framework is applied to a specific problem, $\alpha$ and $\alpha_W$  are typically dependent on the mesh size. The dependency of these constants on relevant problem parameters are discussed in Section 4. Combining these norm equivalences with Assumption \ref{ass1} leads to the following corollary. 
\begin{corollary} 
\label{co:stab}
 Let $\vec{b}\in \mathbb{C}^n$ and  $\vec{x}_u$ be such that $A \vec{x}_u = \vec{b}$. Then there holds that 
\begin{equation*}
| \vec{x}_u |\leq C_{2S} |\vec{b}|.
\end{equation*}
Where $C_{2S} := C_S (\alpha_W \alpha )^{-1}$ . 
\end{corollary}
\begin{proof} Let $q \in V$ be such that 
\begin{equation*}
(q,v)_W = \vec{x}_v^* \vec{b} \quad \forall v \in V,
\end{equation*}
where $(\cdot, \cdot)_W$ is inner product on $W$. Using Cauchy-Schwarz inequality and the norm equivalence given in equation (\ref{ass2}) there holds that $\| q \|_W \leq \alpha_W^{-1} | \vec{b} |$. Via this construction, vector $\vec{b}$ defines an antilinear functional on $W^\prime$ as $L(v) := (q,v)_W$.  By the definition of the dual norm and Cauchy-Schwarz inequality 
\begin{equation}
\label{eq:load_eq}
\| L \|_{W^\prime} = \sup_{w \in W}\frac{ \left | (w,q)_W \right| }{\| w \|_W} \leq \sup_{w \in W}\frac{ \|w\|_W\|q\|_W  }{\| w \|_W} = \|q  \|_W.
\end{equation}
It follows that 
\begin{equation*}
\|L \|_{W^\prime}\leq \alpha_W^{-1} |\vec{b}|. 
\end{equation*}
Combining the above equation with Assumption \ref{ass1} and equation (\ref{ass2}), we obtain
\begin{equation*}
\alpha_W \alpha | \vec{x}_u |\leq C_S |\vec{b}|.
\end{equation*}
\end{proof}

The above Corollary essentially gives a lower bound for the smallest singular value of $A$. There holds that 
\begin{equation*}
\sigma_{min}(A)^{-1} = \min_{\vec{x}\in \mathbb{C}^N} \frac{|A^{-1}\vec{x}|}{ | \vec{x} |}
\end{equation*}
so, that $C_{2S}^{-1} \leq \sigma_{min}(A)$. Corollary \ref{co:stab} can be used to derive exclusion region near the origin. We give here a direct proof that fits well to the framework of the paper. Same result can be established from the lower bound for the smallest singular value by using Theorem 3 from \cite{Ko:01}.
\begin{theorem}
\label{th:exclusion}
Let Assumption \ref{ass1} hold and let $C_{2S}$ be as defined in Corollary \ref{co:stab}. Then there holds that 
\begin{equation*}
B(0, \frac{1}{ C_{2S} }- \epsilon) \subset \Lambda_\epsilon^c,
\end{equation*}
in which $B(z_0,r) := \left \{\; z\in \mathbb{C} \: | \; |z-z_0| < r \; \right\}$.

\end{theorem}
\begin{proof} From the definition of the auxiliary variable (\ref{eq:aux_def}) it follows that 
\begin{equation*}
A \vec{x}_v = z \vec{x}_v - \vec{x}_u 
\end{equation*}
Application of Corollary \ref{co:stab} gives
\begin{equation*}
|\vec{x}_v | \leq C_{2S} \left( |z|  |\vec{x}_v | + | \vec{x}_u | \right)
\end{equation*}
i.e.
\begin{equation}
\label{eq:inclusion_bound}
|\vec{x}_v | \leq \frac{ C_{2S}  }{ 1 - C_{2S}  |  z | } | \vec{x}_u |.
\end{equation}
When $|z| < C_{2S}^{-1}$, the above bound implies that $(zI-A)$ is non-singular. In this case, combining equations (\ref{eq:inclusion_bound}) and (\ref{eq:norm_aux}) gives 
\begin{equation*}
|(zI - A )^{-1}|\leq  \frac{ C_{2S}  }{ 1 - C_{2S}  |  z | }.
\end{equation*}
To obtain the exclusion region, we set  
\begin{equation*}
\frac{ C_{2S}  }{ 1 - C_{2S}  |  z | } < \epsilon^{-1},
\end{equation*}
which gives the bound
\begin{equation*}
|z| < \frac{1}{C_{2S}}- \epsilon.
\end{equation*}
\end{proof}

The inclusion is obtained by relating pseudospectrum to FOV. The following Theorem is proven e.g. in,  \cite{Tr:97}. For completeness, we give a proof using the notation used in this Section. 
\begin{theorem} 
\label{th:FOV_2_PS} Let $S_\epsilon := \left \{ \; z \in \mathbb{C} \; | \; \mbox{dist}(z,FOV(A) ) \leq \epsilon \; \right\}$ in which
\begin{equation*}
\mbox{dist}(z, Q) := \inf_{ q \in Q } |z - q |. 
\end{equation*}
Then there holds that $\Lambda_\epsilon \subset S_\epsilon$.
\end{theorem}
\begin{proof} The auxiliary variable is defined as 
\begin{equation*}
(A - zI ) \vec{x_v} = \vec{x}_u.
\end{equation*}
\noindent Testing the above equation with $\vec{x}_v$ gives
\begin{equation*}
\vec{x}_v^*A \vec{x}_v - z \vec{x}_v^* \vec{x}_v = \vec{x}_v^* \vec{x}_u.
\end{equation*}
Using Cauchy-Schwarz inequality gives
\begin{equation*}
| \vec{x}_v | | \vec{x}_u | \geq |\vec{x}_v^*A \vec{x}_v - z \vec{x}_v^* \vec{x}_v |  = \vec{x}_v^*\vec{x}_v \left|\frac{\vec{x}_v^*A \vec{x}_v}{\vec{x}_v^*\vec{x}_v} - z  \right|.
\end{equation*}
This is, 
\begin{equation*}
|\vec{x}_v| \left|\frac{\vec{x}_v^*A \vec{x}_v}{\vec{x}_v^*\vec{x}_v} - z  \right| \leq |\vec{x}_u|.
\end{equation*}
By the definition of FOV(A) in equation (\ref{eq:FOV}) there holds that 
\begin{equation*}
| \vec{x}_u | \geq \mbox{dist}(z,FOV(A) ) | \vec{x}_v |. 
\end{equation*}
\end{proof}

Theorem \ref{th:FOV_2_PS} gives tools for deriving an inclusion for the pseudospectrum. The FOV is directly related to the boundedness properties of the sesquilinear form of the original problem. This relation arises from the connection $\vec{x}_v^* A \vec{x}_v = a(v,v)$. The simplest  estimate follows from boundedness of the sesquilinear form. Assume that there exists $C>0$ such that
\begin{equation*}
| a(u,u) | < C \|u \|^2_V \quad \forall u \in V.  
\end{equation*}
Then there holds that 
\begin{equation*}
FOV(A) \subset B(0,C).
\end{equation*} 
This is a very crude estimate, but it demonstrates how FOV can be bounded in simple cases. However, as we will see, more refined estimates are required to avoid inclusion of zero to the approximate pseudospectrum.

\section{Examples} 

In this section, we demonstrate the presented theory with three examples. In all examples, we assume that $\Omega \subset \mathbb{R}^d,d=2,3$ is a bounded domain with Lipschitz continuous boundary. We use standard notation for Sobolev spaces, see \cite{Br:2007}. 

The finite element space $V$ is defined as 
\begin{equation}
V := \{u \in H^1(\Omega) \; | \; u \in P^1(K) \quad \forall K \in \mathcal{T} \; \},
\end{equation}
where $\mathcal{T}$ is a shape regular triangular or tetrahedral partition of $\Omega$, \cite{Br:2007}. This is, $V$ is the space of first order Lagrange finite elements. The space of first order polynomials over set $K$ is denoted by $P^1(K)$ and the mesh-size by $h$ , respectively.

The presented theoretical results are independent of the domain, but the actual numerical examples are computed on $\Omega = (-1,1)^2 \setminus (0,1)^2$. The meshes used in the tests are generated from a coarse mesh with approximately 100 nodes using uniform refinement. The coarse mesh is called level one mesh, once refined coarse mesh as a level two mesh and so on. 

Throughout this Section, $c,C > 0$ are generic positive constants independent of mesh size $h$, solution, load, and parameters of the weak problem, if not otherwise stated. They may depend on  the shape regularity constant of the partition $\mathcal{T}$ and the domain $\Omega$. 

\subsection{Poisson equation} We begin by considering the finite element discretization of the Poisson equation: Find $u \in V_0$ such that 
\begin{equation}
\label{eq:weak_poisson}
(\nabla u, \nabla v) = (f,v) \quad \forall v \in V_0. 
\end{equation}
In which $V_0 = V \cap H^1_0(\Omega)$ and $f \in L^2(\Omega)$. This is
\begin{equation*}
a(u,v) :=  (\nabla u, \nabla v)  \mbox{  and  } L(v) := (f,v)
\end{equation*}
so that $L \in (L^2(\Omega))^\prime$. We use the standard $H^1$-norm
\begin{equation*}
\| u \|_1^2 := (\nabla u, \nabla u) +(u,u)
\end{equation*}
for the space $V_0$. 

It is straightforward  to see that the matrix $A$ related to problem (\ref{eq:weak_poisson}) is symmetric and positive definite, \cite{Br:2007}. The convergence of iterative methods for such linear systems can be analyzed using much easier techniques than pseudospectrum. However, such a simple example is useful for demonstrating what kind of information on GMRES convergence can be obtained based on the inclusion and exclusion results. 

Pseudospectrum of a normal matrix can be easily computed from it's eigenvalues. All normal matrices are unitary diagonalizable, hence there exists a diagonal $D \in \mathbb{C}^{N \times N}$and a unitary $Q \in \mathbb{C}^{N \times N}$ such that $A = Q^* D Q$. Based on this expansion, there holds that 
\begin{equation*}
\left| (zI -A)^{-1}\right | = \left | (z -D)^{-1}\right | = \max_{\lambda \in \sigma(A)} | (z - \lambda)^{-1} |.
\end{equation*}
Thus, pseudospectrum of any normal matrix is a union of discs centered around it's eigenvalues $\lambda_i$, 
\begin{equation*}
\Lambda_\epsilon = \cup_{i=1}^N B( \lambda_i, \epsilon ).
\end{equation*}
The pseudospectrum for level one mesh is visualized in Figure \ref{fig:ps_poisson} for different values of $\epsilon$. 

\begin{figure}[ht]
\label{fig:ps_poisson}
\includegraphics[scale=0.6]{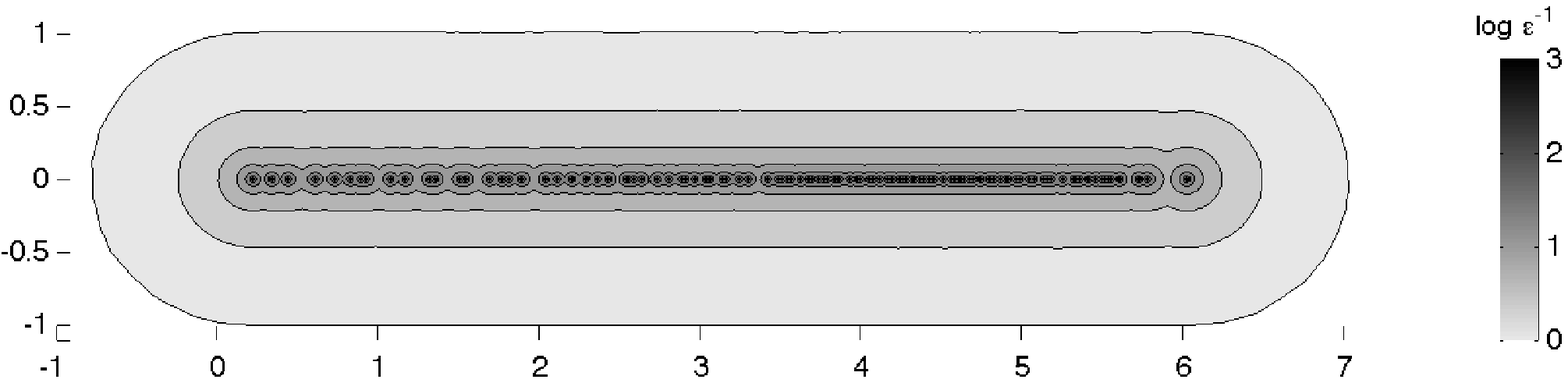} \\
\includegraphics[scale=0.6]{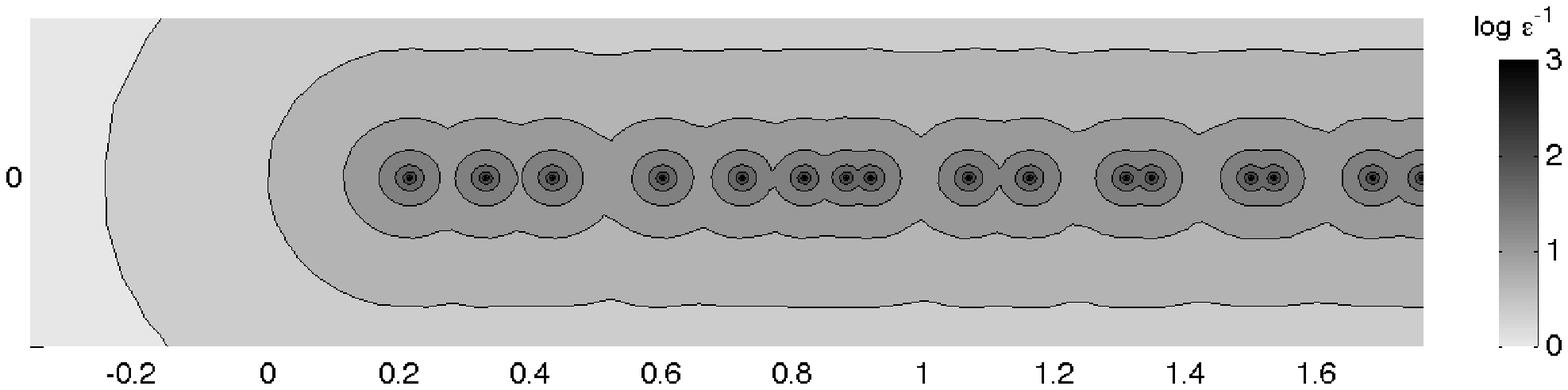}
\caption{Pseudospectral set for the Poisson problem on level one mesh. For sufficiently small $\epsilon$, the set is composed of disjoint disks with radius $\epsilon$.}
\end{figure}

Next, we derive inclusion and exclusion regions using Theorem \ref{th:exclusion} and \ref{th:FOV_2_PS}. First, we need to establish a stability estimate satisfying Assumption \ref{ass1}. As we are interested in mesh size explicit bounds, we use $h$-explicit norm equivalences instead of equation (\ref{ass2}). For the Poisson problem, stability estimate follows from the weak problem (\ref{eq:weak_poisson}) by using Poincare-Friedrichs inequality. Let $u \in V_0$ be the solution to (\ref{eq:weak_poisson}) then there exists a constant $C>0$ such that
\begin{equation*}
\| u \|_1 \leq C \| f \|_0.
\end{equation*}
Following this stability estimate, we choose the space $W$ as $L^2(\Omega)$ and $\| \cdot \|_W = \|\cdot \|_0$. Norm equivalences between $H^1(\Omega)$-, $L^2(\Omega)$- and the Euclidian norm can be derived in the finite element space $V$ using the scaling argument and inverse inequality, \cite{QuVa:94}. There exists $c$ and $C$  such that
\begin{equation}
\label{eq:L2_2_EUCLIDIAN}
c h^{d/2}|\vec{x}_u | \leq \|u \|_0 \leq C h^{d/2} | \vec{x}_u | \quad \forall u \in V
\end{equation}
and
\begin{equation}
\label{eq:H1_2_EUCLIDIAN}
c h^{d/2}|\vec{x}_u | \leq \|u \|_1 \leq  C h^{d/2-1}  | \vec{x}_u | \quad \forall u \in V.
\end{equation}

Now, we can apply Corollary \ref{co:stab} to derive a stability constant for the linear system arising from the weak problem (\ref{eq:weak_poisson}). Let $\vec{x}_u$ be such that $A \vec{x}_u = \vec{b}$. Then by Corollary \ref{co:stab} and the $h$-explicit norm equivalences, there  exists a constant $C$ such that  
\begin{equation*}
| \vec{x}_u | \leq C h^{-d}| \vec{b}|.
\end{equation*}
Application of Theorem \ref{th:exclusion} gives the following exclusion near the origin, 
\begin{equation*}
B(0, Ch^{d} - \epsilon )\subset \Lambda_\epsilon^c. 
\end{equation*}
We proceed by deriving an inclusion for $FOV(A)$, which together with Theorem \ref{th:FOV_2_PS} gives inclusion for $\Lambda_\epsilon$. It is easy to derive the estimates
\begin{equation*}
\Im{a(u,u)} = \Im \|\nabla u \|_0^2 = 0 \quad \forall u \in V_0
\end{equation*}
and 
\begin{equation*}
c h^{d}  | \vec{x}_u |^2 \leq \Re a(u,u) < C h^{d-2} |\vec{x}_u |^2 \quad \forall u \in V_0.
\end{equation*}
So that $FOV(A) \subset \left \{\; x \in \mathbb{R} \; | \;  c h^{d}   < x <  C h^{d-2} \; \right\}$. An application of Theorem \ref{th:FOV_2_PS} gives the inclusion $\Lambda_\epsilon \subset \tilde{S}_\epsilon$, in which
\begin{equation*}
\tilde{S}_\epsilon := \{ \; z \in \mathbb{C} \; | \; \mbox{dist} \left( z, \left \{\; x \in \mathbb{R} \; | \;  c h^{d}   < x <  C h^{d-2} \; \right\} \right) \leq \epsilon \; \}.
\end{equation*}

The above inclusion and exclusion regions give us an approximation of pseudospectrum $\tilde{\Lambda}_\epsilon$,  
\begin{equation*}
\tilde{\Lambda}_\epsilon := \tilde{S}_\epsilon \setminus B(0, C_1 h^{d} - \epsilon ).
\end{equation*}
Where the constant $C_1>0$ is independent of $h$ and $\epsilon$ . To exclude the origin from this approximate pseudospectrum, we have to choose the parameter $\epsilon$ as $\epsilon \leq C_1 h^{d} $. In this case, the length of the boundary curve around the approximate pseudospectrum  satisfies $| \partial \tilde{\Lambda}_\epsilon | \leq C_2 h^{d-2}$ for some $C_2>0$  independent of $h$ and $\epsilon$.  When combined with equation (\ref{eq:GMRES_B2}) approximate pseudospectrum gives the GMRES convergence bound
\begin{equation}
\label{eq:GMRES_poisson1}
| \vec{r}_i| \leq  \frac{C_2 h^{d-2}}{2 \pi \epsilon}  \inf_{ \substack{p \in \tilde{P}_i  \\ p(0)=1}} \sup_{z \in \tilde{\Lambda_\epsilon} } |p (z) |  | \vec{r}_0 | \quad \forall \epsilon \leq C_1 h^d
\end{equation}
The set $\tilde{\Lambda}_\epsilon$ can be covered either with an ellipsoid or a circle and the minimization problem can be solved using estimates given in \cite{Gr:1997,Sa:2003}. There holds that 
\begin{equation*}
\inf_{ \substack{p \in \tilde{P}_i  \\ p(0)=1}} \sup_{z \in B(c,r) } |p (z) | \leq \left( \frac{r}{|c|} \right)^{i}
\end{equation*}
Although the estimate could be optimized with respect to parameter $\epsilon$, we have chosen $\epsilon = 0.5 C_1 h^{d}$, which gives correct asymptotic behavior with respect to $h$. Using this $\epsilon$ and  $c = C_2 h^{d-2}$, the circle based bound leads to the estimate 
\begin{equation*}
| \vec{r}_i| \leq \frac{  C_2 h^{-2} }{\pi C_1 } \left( \frac{ 1 }{ 1 + \frac{C_1}{2 C_2} h^{2}} \right)^i | \vec{r}_0 |
\end{equation*}
When the termination criteria for GMRES is chosen such that the relative residual satisfies $|\vec{r}_i | |\vec{r}_0|^{-1} \leq tol$, the above estimate gives the required number of iterations $N$ as
\begin{equation}
\label{eq:N_poisson1}
N \approx - \frac{2 C_2}{ C_1} h^{-2} \left( \log{ tol }  -  \log{ \frac{  C_2 h^{-2} }{\pi C_1} } \right)
\end{equation}

Our approximate pseudospectrum cannot capture the behavior of $\Lambda_\epsilon$ for very small values of $\epsilon$. For example in the current case, the exact pseudospectrum is composed of small discs  with boundary length $2 \pi \epsilon$. Let $\epsilon$ be such that the discs generating the pseudospectrum do not intersect. Any finite union of disjoint disks satisfies the conditions placed on the set $U$ in the Dunford integral. Using equation (\ref{eq:dunford}) we obtain the estimate
\begin{equation*}
|p(A)| \leq \frac{1}{2 \pi \epsilon} \sum_{i=1\ldots N_0}  | \partial B(\epsilon, \lambda_i) | \sup_{ z \in B(\epsilon, \lambda_i) }|p(z) |
\leq  N_0 \sup_{ z \in \Lambda_\epsilon }|p(z) |,
\end{equation*}
which is valid for sufficiently small $\epsilon$. Here $N_0$ is the number of disjoint eigenvalues of $A$. For quasi-uniform meshes, there exists $C$ such that $N_0 \leq C h^{-d}$ so that 
\begin{equation*}
|p(A)| \leq  C h^{-d} \sup_{ z \in \Lambda_\epsilon }|p(z) |. 
\end{equation*}
Combining the above estimate with equation (\ref{eq:GMRES_B1}) gives 
\begin{equation}
\label{eq:refined_estimate_poisson}
\frac{| \vec{r}_i |}{ | \vec{r}_0|} \leq C h^{-d} \inf_{ \substack{p \in \tilde{P}_i  \\ p(0)=1}} \sup_{ z \in \Lambda_\epsilon }|p(z) |.
\end{equation}
This estimate based on the exact set $\Lambda_\epsilon$  has a different multiplicative term in comparison to equation (\ref{eq:GMRES_poisson1}). Interestinly, for $d=1$, multiplicative term is smaller, for $d=2$ it is equivalent and for $d=3$ bigger. Regardless of the  multiplicative constant, the estimate (\ref{eq:refined_estimate_poisson}) can deliver improved convergence number estimates. The best possible bound can be obtained at the limit $\epsilon = 0$, when the minimization problem can be solved using Chebychev polynomials, see e.g. \cite{Sa:2003}. Based on the FOV, the condition number $\kappa(A) \leq Ch^{-2}$. We obtain an estimate for the number of iterations
\begin{equation}
\label{eq:N_poisson2}
N \approx - C h^{-1} ( \log(tol) - \log(Ch^{-d}))
\end{equation} 
 
The main difference between the estimates (\ref{eq:N_poisson1}) and (\ref{eq:N_poisson2}) is in in the power of the mesh size $h$. For the particular problem, this difference is due to the fact, that the set $\tilde{\Lambda}_\epsilon$ cannot capture the behaviour of the pseudospectrum for small $\epsilon$. For complicated problems, such knowledge is very difficult to come by and one has to be satisfied with worst case estimates, such as equation (\ref{eq:N_poisson1}). The second difference between the two estimates is in the additive terms. These additive terms are relevant only when tolerance is of the same order of magnitude with $Ch^{-d}$, which requires usage of very fine mesh sizes 

\subsection{ Helmholtz equation with absorbing boundary conditions}

The Helmholtz equation with first-order absorbing boundary conditions is a more realistic example for the analysis presented in this paper. The weak problem reads: Find $u \in H^1(\Omega)$ such that 
\begin{equation}
\label{eq:weak_abs}
a(u,v) = L(v) \quad \forall v \in H^1(\Omega).
\end{equation}
in which 
\begin{equation}
\label{eq:abs_forms}
a(u,v) := (\nabla u, \nabla v) + \mathrm{i} \kappa \left( u,v \right)_{\partial \Omega} - \kappa^2 (u,v ) \mbox{ and } L(v) := (f,v) + \left( g,v \right)_{\partial \Omega}.
\end{equation}
The parameter $\kappa \in \mathbb{R}, \kappa >0$, $f \in L^2(\Omega)$ and $g \in L^2(\partial \Omega)$. The inner product $(\cdot,\cdot)_{\partial \Omega}$ is the standard $L^2$-inner product over $\partial\Omega$. The stability of this problem has been analyzed in domains excluding any resonant behavior, \cite{Me:95}. 
\begin{theorem} 
Let $\Omega$ be a bounded, star shaped domain with a smooth boundary and let $u \in H^1(\Omega)$ be the solution to problem (\ref{eq:weak_abs}). Then there exists a constant $C_S > 0$ independent of $u$,$f$,$g$ and $\kappa$ such that 
\label{th:stab_abs}
\begin{equation*}
\|u \|_\kappa \leq C_S \left( \|f \|_0 + \| g \|_{0,\partial \Omega} \right),
\end{equation*} 
in which the norm $\| \cdot \|_\kappa$ is defined as 
\begin{equation}
\label{eq:kappa_norm}
\| u \|_\kappa^2 := \|\nabla u \|_0^2 + \kappa^2 \|u \|_0^2.
\end{equation}

\end{theorem}

The finite element approximation $u_h$ is defined as: Find $u_h \in V$ such that 
\begin{equation*}
\label{eq:abs_discrete}
a(u_h,v) = (f,v) + (g,v)_{\partial \Omega} \quad \forall v \in V. 
\end{equation*}
When the solution has $H^2(\Omega)$-regularity, the existence of a unique solution to this problem can be guaranteed, when the mesh size requirement $\kappa^2 h << 1$ is satisfied, \cite{BaIh:95,BaIh:97,Me:95}. In this case, there exists a constant C such that the a priori error estimate 
\begin{equation}
\label{eq:abs_aprior}
\|u - u_h \|_\kappa \leq C h \left( \| f \|_0 +  \| g \|_{0,\partial \Omega} \right). 
\end{equation}
holds. 

Due to the boundary term $\mathrm{i} \kappa \left( u, v \right)_{\partial \Omega}$, problem (\ref{eq:weak_abs}) leads to a linear system with a non-normal coefficient matrix. As the boundary term depends on $\kappa$, it is complicated to determine if the non-normality is meaningful or not. In addition, due to the relation between the wave-number and the mesh size it is difficult to study the asymptotic behaviour of GMRES, when $\kappa$ tends to infinity.
\begin{figure}[ht]
\includegraphics[scale=0.8]{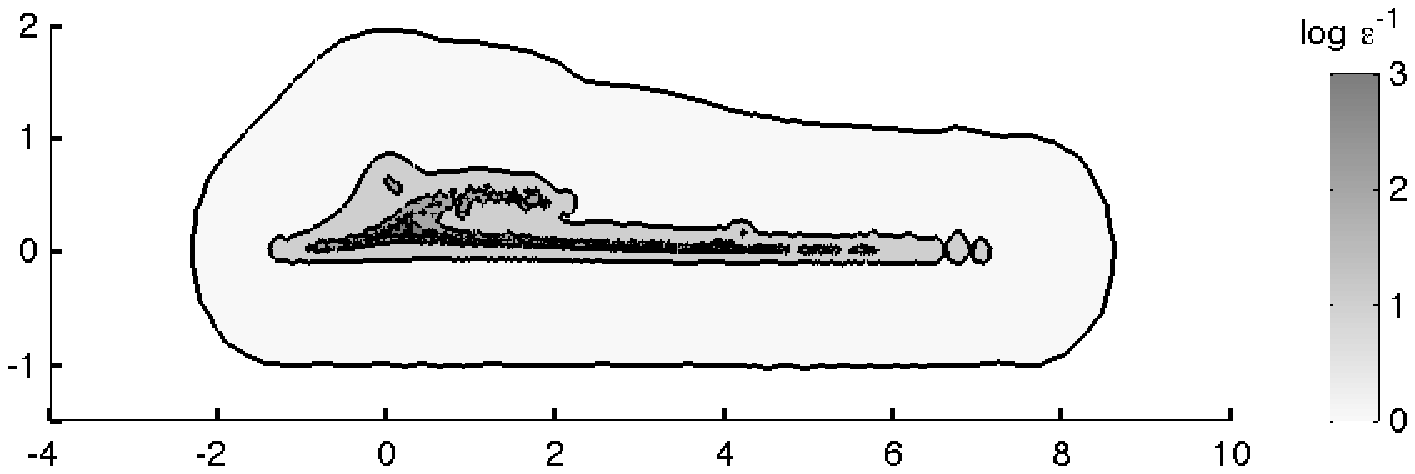} 
\includegraphics[scale=0.8]{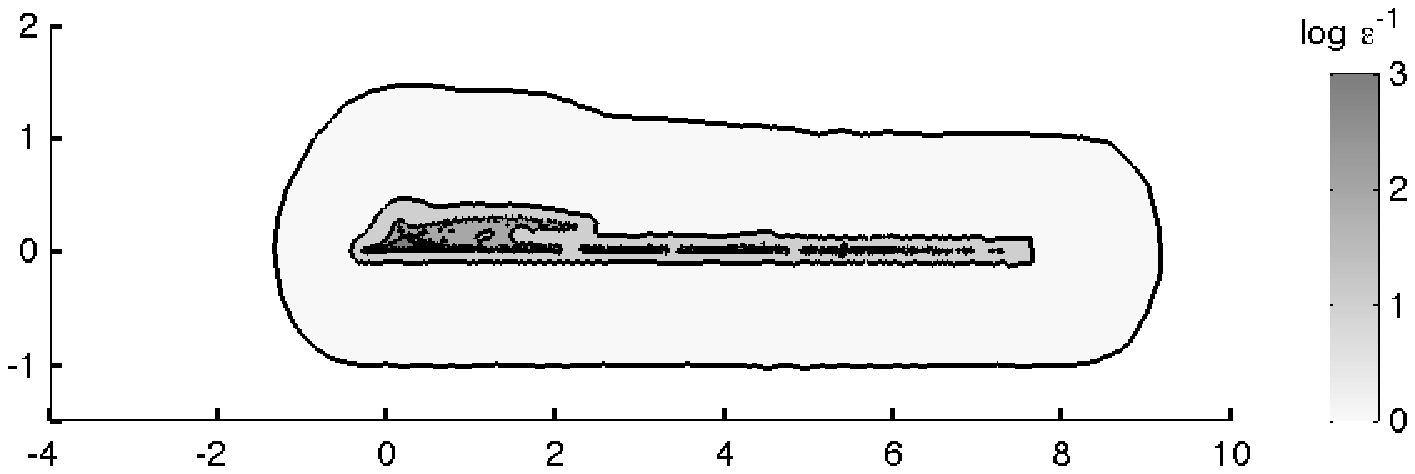} 
\caption{The pseudospectral set for the Helmholtz equation with first order absorbing boundary conditions. The parameter $\kappa = 8 \pi$ and level three mesh is used in the upper figure and level four in the lower one.}
\label{fig:abs1}
\end{figure}

\begin{figure}[ht]
\includegraphics[scale=0.45]{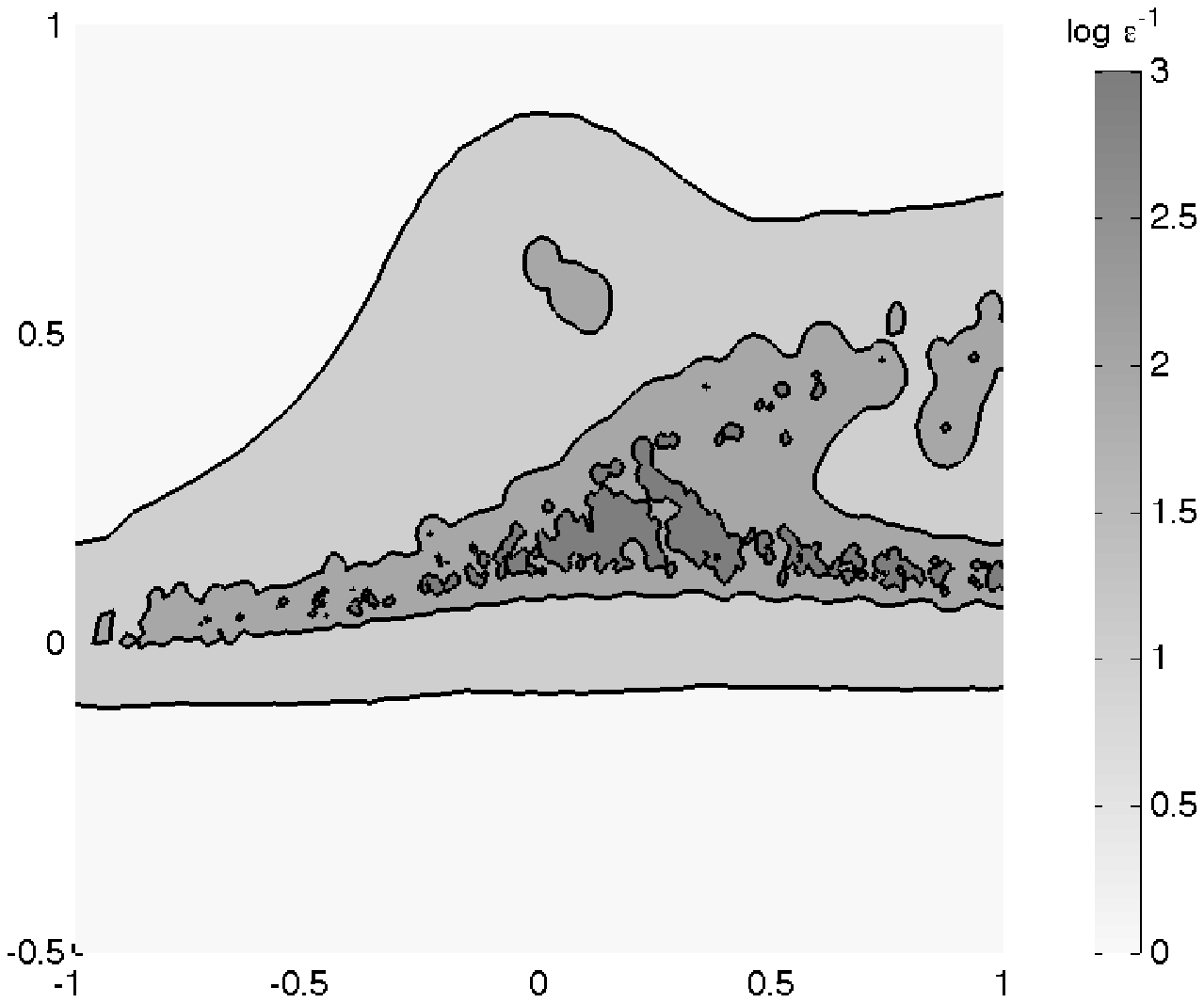}
\includegraphics[scale=0.45]{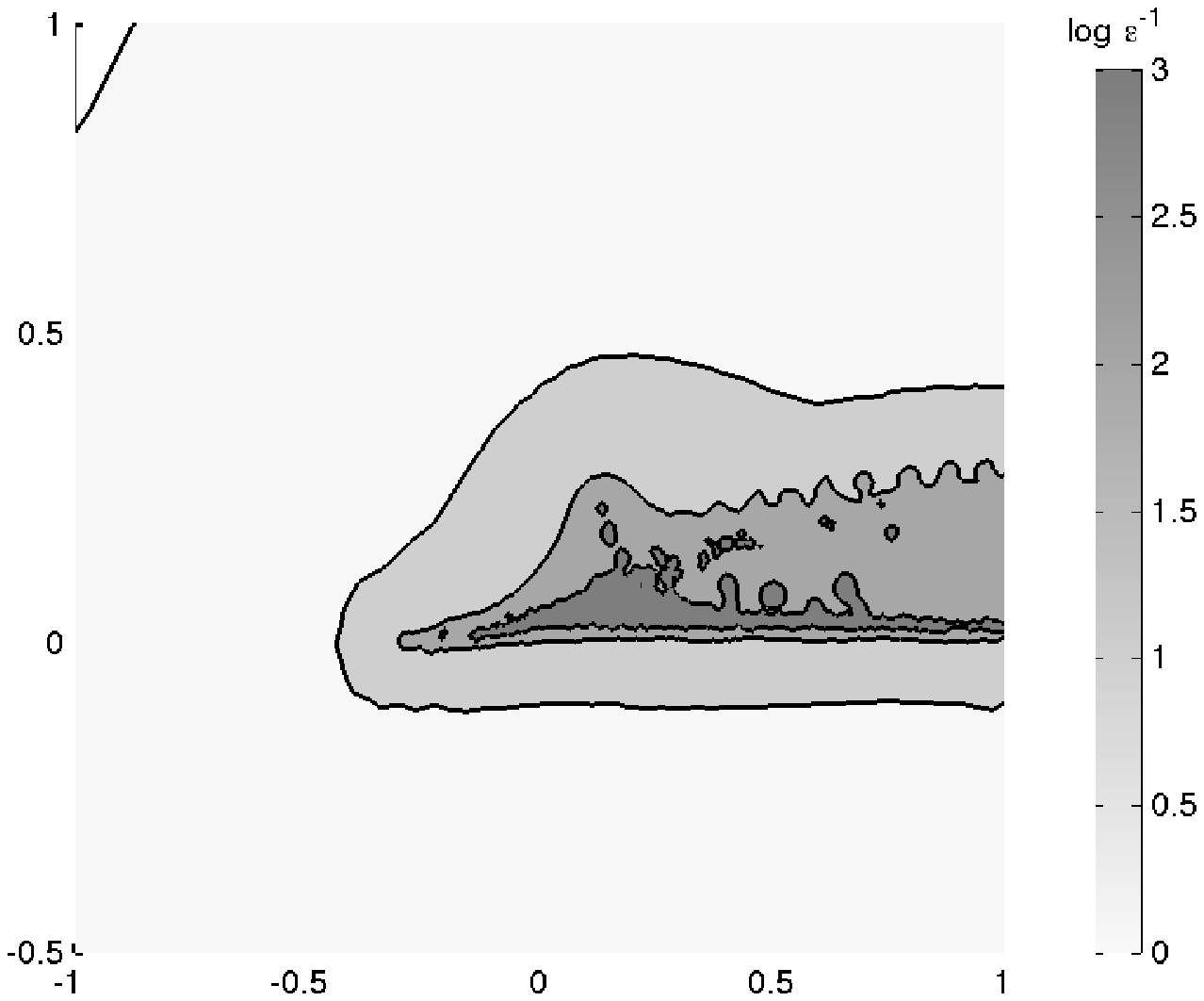}
\caption{Pseudospectral set for the Helmholtz equation with first order absorbing boundary conditions. The parameter $\kappa = 8 \pi$. The mesh levels is three on left and four on right. One can observe the convergence of the set when mesh size tends to zero.  }
\label{fig:abs2}
\end{figure}

When the mesh size is sufficiently small so that the a priori error estimate (\ref{eq:abs_aprior}) holds, Theorem \ref{th:stab_abs}  implies  stability of the discrete problem. We obtain 
\begin{equation}
\label{eq:abs_fem_stab}
\| u_h \|_\kappa \leq  (1+Ch) ( \| f \|_0 + \| g \|_{0,\Omega}) \leq C  ( \| f \|_0 + \| g \|_{0,\Omega}).
\end{equation}
This discrete stability estimate holds under the following assumptions. 
\begin{assumption} 
\label{ass3} 
Assume that $\Omega$ is a bounded, star shaped domain with a smooth boundary, the mesh size $h$ is such that $\kappa^2 h << 1$ and the solution $u$ to problem (\ref{eq:weak_abs}) has $H^2(\Omega)$-regularity.
\end{assumption}

Following the discrete stability result (\ref{eq:abs_fem_stab}) we choose the space $W $ as $L^2(\Omega)$ with the norm $\| \cdot\|_W = \|\cdot \|_0$. The $h$- explicit norm equivalences given in equation (\ref{eq:L2_2_EUCLIDIAN}) can be used for this space. As we are interested in wavenumber and the mesh size explicit estimates, we use the $\kappa$-dependent norm given in equation (\ref{eq:kappa_norm}) for the space V. Norm equivalences for this $\kappa$-dependent norm are easily established using equation (\ref{eq:L2_2_EUCLIDIAN}) and (\ref{eq:H1_2_EUCLIDIAN}) as
\begin{equation}
\label{eq:K_2_EUCLIDIAN}
c \kappa h^{d/2}  | \vec{x}_u | \leq \| u\|_\kappa \leq C( h^{d/2-1}+\kappa h^{d/2} ) | \vec{x}_u |  \quad \forall u \in V,
\end{equation}
for some $c,C$. Application of Corollary \ref{co:stab} gives the stability estimate for the coefficient vector 
\begin{equation}
\label{eq:abs_stab}
| \vec{x}_{u_h} | \leq C \frac{h^{-d}}{\kappa}| \vec{b}|.
\end{equation}
Using Theorem \ref{th:exclusion}  leads to the exclusion region 
\begin{equation*}
B(0, C \kappa h^{d}- \epsilon) \subset \Lambda_\epsilon^c
\end{equation*}
around the origin. To obtain an inclusion, we again derive an inclusion for $FOV(A)$ and apply Theorem \ref{th:FOV_2_PS}. The sesquilinear form satisfies the boundedness estimates
\begin{equation*}
|\Re a(u,u)| \leq C \| u\|_\kappa^2 \quad \forall u \in V
\end{equation*}
and
\begin{equation*}
0 \leq \Im a(u,u) \leq C \kappa h^{d-1} | \vec{x}_u |^2 \quad  \forall u \in V 
\end{equation*}
for some $C$. The estimate between the $L^2(\partial\Omega)$- and Euclidian norm used in above is derived using identical techniques as used for proving inequality (\ref{eq:L2_2_EUCLIDIAN}). When Assumption \ref{ass3} is satisfied, combining the two boundedness estimates leads to the inclusion
\begin{equation*}
FOV(A) \subset \left \{ \; z \in \mathbb{C} \; | \; |z| \leq C \mbox{ and } \; 0 \leq \Im{z} \leq C \kappa h^{d-1} \; \right\}.
\end{equation*}
This set contains the origin, so it cannot be used to derive GMRES convergence bounds. In this case, the presented theory is genuinely required to understand GMRES convergence. 

To validate the derived inclusion and exclusion regions, we have computed examples from the exact set $\Lambda_\epsilon$ for $\kappa = 8 \pi$ using mesh levels three and four. The results are visualized in Figures \ref{fig:abs1} and \ref{fig:abs2}. Although, the $L$ - shaped domain used in computations does not have smooth boundary nor $H^2(\Omega)$-regularity, the actual pseudospectral set is in good agreement with our theoretical results. Most importantly, when $\epsilon$ is sufficiently large, the pseudospectrum curls around the origin as predicted. Due to the solution having less that $H^2(\Omega)$-regularity, the requirement on the mesh size on $L$-shaped domain just takes the form $h^\alpha \kappa << 1$, for some $\alpha <2$, depending on regularity of the exact solution. 

%
%

The approximate pseudospectrum could also be used to to derive convergence estimate for GMRES method using Bratwurst shaped domains to solve the minimization problem. However, as a preconditioner would always be applied, the current case is not very interesting hence we do not proceed further with it. 


\subsection{Shifted-Laplace preconditioned Helmholtz equation}
The analysis of inclusion and exclusion regions is more complicated, when a preconditioner is applied to speed up the convergence of the GMRES method. Several different preconditioners have been proposed for problem (\ref{eq:weak_abs}), see e.g.  \cite{Er:08}. We consider here the shifted-Laplace preconditioner \cite{ErOoVu:04}. This preconditioner is based on solving an auxiliary problem on each step of the iteration. The auxiliary problem is defined as: For a given $u \in V$ find $Pu \in V$ such that 
\begin{equation}
\label{eq:prec_action}
b( Pu,v) = \vec{x}_v^* \vec{x}_u \quad \forall v \in V. 
\end{equation}
The sesquilinear form $b$ in the above equation is given as 
\begin{equation*}
b(u,v) = (\nabla  u, \nabla v) + \mathrm{i}\kappa \left(u, v \right)_{\partial\Omega} - \kappa^2 (u,v)  + \mathrm{i}{\sigma} (u,v),
\end{equation*}
in which $\kappa$ is  as defined in Section 4.2 and $\sigma \in \mathbb{R}, \sigma >0$. This is, a loss term $\mathrm{i}\sigma(u,v)$ is added to the sesquilinear from defined in equation (\ref{eq:abs_forms}). The addition of the loss term leads to a stability estimate on the finite element space $V$ independent of the mesh size. Choosing $v=Pu$ in equation (\ref{eq:prec_action}) gives%
\begin{equation*}
b(Pu,Pu) = \vec{x}_{Pu}^* \vec{x}_u  \quad \forall v \in V.
\end{equation*}
Taking imaginary part leads to
\begin{equation*}
\kappa \| P u \|_{\partial \Omega} + \sigma \| Pu \|_0^2 = \Im  \vec{x}_{Pu}^* \vec{x}_u.
\end{equation*}
This is, 
\begin{equation*}
\sigma \|P u \|_0^2 \leq \Im  \vec{x}_{Pu}^* \vec{x}_u.
\end{equation*}
Now, using Cauchy-Schwarz inequality and norm equivalence (\ref{eq:L2_2_EUCLIDIAN}) gives
\begin{equation} 
\label{eq:stab_abs_loss}
\| Pu \|_0 \leq C \sigma^{-1} \| u \|_0 \quad \forall u \in V
\end{equation}
for some $C$. The matrix form of the preconditioner is denoted as $B^{-1}$, where $B_{ij} = b(\varphi_j, \varphi_i)$. Hence, the problem to be solved by the GMRES method is
\begin{equation*}
AB^{-1}\tilde{ \vec{x} } = \vec{b} \quad ,  \quad \vec{x} = B^{-1}\tilde{\vec{x}}. 
\end{equation*}

The rationale behind using shifted-Laplace preconditioners is that when a sufficiently large loss term is added, the action of the preconditioner can be efficiently evaluated using a multigrid method, \cite{ErOoVu:06}. When applied directly  to solve the original problem (\ref{eq:weak_abs}), multigrid methods face two challenges, \cite{El:01}. The standard smoothing iteration is not stable and the coarse grid correction has to be made on a sufficiently fine mesh. The introduction of a loss term has been analyzed in \cite{Ha:14} for a problem with zero Dirichlet boundary conditions. In this case, additional losses improve the multigrid solver by allowing the coarse grid correction to be made on a coarser mesh. The coarse grid depends on the loss term, hence there is a tradeoff between the number of GMRES iterations and the cost of applying the preconditioner. Typically, the loss parameter is chosen as $\sigma = 0.5 \kappa^2$.  For simplicity, we will consider here only the exact preconditioner. This gives good insight on what one can expect from the inexact case.

As we will see, a shifted-Laplace preconditioner can eliminate the mesh size dependency from the pseudospectral set. This is, the inclusion and exclusion regions are independent of the applied mesh size. This is a desired property, as the mesh size dependency in the non-preconditioned case leads quickly to an unbearably large number of iterations. The exclusion regions will, however depend on the ratio of $\kappa$ and $\sigma$. 

The shifted-Laplace preconditioner has been previously analyzed in \cite{GiErVu:07} by estimating the location of the eigenvalues. The existing analysis is not explicit in $\sigma$ and does not take the non-normality into account. In addition, the previous work does not include the exclusion region around the origin, which we can obtain using Theorem 3.1. and the stability result given in equation (\ref{eq:abs_fem_stab}). 

To study the shifted-Laplace preconditioner, we interpret the matrix $AB^{-1}$ as the matrix form of the sesquilinear form $a(Pu,u)$, where $a(u,v)$ is as defined in equation (\ref{eq:abs_forms}) and $P$ in equation (\ref{eq:prec_action}). A suitable stability estimate for this sesquilinear form is established by the following Corollary.

\begin{corollary} 
\label{eq:sl_corollary}
Let $u \in V$ be such that
\begin{equation}
\label{eq:sl_proof}
a(Pu,v) = (f,v) \quad \forall v \in V.
\end{equation}
In addition, let Assumption \ref{ass3} be satisfied. Then there exists a constant $C>0$ independent of $u$,$f$,$\kappa$,$h$ and $\sigma$ such that 
\begin{equation*}
| \vec{x}_u| \leq   C h^{d/2} \left( 1 +  \frac{\sigma}{\kappa}  \right) \ \|f\|_0. 
\end{equation*}

\end{corollary}
\begin{proof}
Application of equation (\ref{eq:abs_fem_stab}) gives
\begin{equation}
\label{eq:sl_proof1}
\|Pu \|_\kappa \leq C \|f \|_0 .
\end{equation}
It follows from definition (\ref{eq:prec_action}) that 
\begin{equation*}
a(Pu,u) = | \vec{x}_u |^2 - \mathrm{i}\sigma (Pu,u).
\end{equation*}
Combining above with equation (\ref{eq:sl_proof}) gives 
\begin{equation}
| \vec{x}_u |^2 =  (f,u) + \mathrm{i}\sigma (Pu,u).
\end{equation}
Using Cauchy-Schwarz inequality, estimate (\ref{eq:sl_proof1}) and norm equivalence (\ref{eq:L2_2_EUCLIDIAN}) gives
\begin{equation*}
| \vec{x}_u |   \leq C h^{d/2} \left( \| f\|_0+ C_S \frac{\sigma}{\kappa}\|f \|_0 \right).
\end{equation*}
\end{proof}

The above stability estimate is given in the norm $\|u \| = | \vec{x}_u|$. Hence, we choose this as the norm of the space $V$. The above Corollary also suggest to choose the space $W=L^2(\Omega)$ as previously. With these choices, a direct application of Theorem \ref{th:exclusion} gives the exclusion 
\begin{equation}
\label{eq:sl_exclusion}
B(0, C \frac{\kappa}{\kappa + \sigma}  - \epsilon) \subset \Lambda_\epsilon^c. 
\end{equation}
When $\sigma = 0$, the preconditioner solves the problem exactly and $\Lambda_{\epsilon} = B(1,\epsilon)$. As the constant in above is $C$ is independent of $\sigma$ and $\kappa$, setting $\sigma = 0$, leads to  $C \leq 1$. A field of values based inclusion can be obtained as follows. There holds that 
\begin{equation*}
\vec{x}_u^*A B^{-1}\vec{x}_u^* = a(Pu,u) = \vec{x}_u^* \vec{x}_u - \mathrm{i}\sigma (Pu,u). 
\end{equation*}
An inclusion for FOV  follows by estimating the last term. By the stability result given in equation (\ref{eq:stab_abs_loss}) and norm equivalence (\ref{eq:L2_2_EUCLIDIAN}), there holds that
\begin{equation*}
\sigma (Pu,u) \leq \sigma\| Pu \|_0 \|u \|_0 \leq C  \vec{x}_u^*\vec{x}_u.
\end{equation*}
This is, the FOV is located inside the set $|1-z|\leq C_1$.  

The polynomial minimization problem in the GMRES convergence bound (\ref{eq:GMRES_B2}) does not give any information on the convergence, when the approximate pseudospectrum is an annulus surrounding the origin. To apply the FOV based estimate, one has to explicitly know the constants in derived inclusion and exclusion regions to guarantee that this cannot happen. The constant $C_1$ in the inclusion for FOV is related to the norm equivalence between $L^2(\Omega)$ and Euclidian norm. It is easy to see, that $C_1 = \sqrt{\mbox{cond}(M)}$, where $M_{ij} = (\varphi_i,\varphi_j)$ is the mass matrix. In typical cases $C_1 \approx 4$, so that derived inclusion  is not useful when $\sigma = 0.5 \kappa^2$ and the dimension of the exclusion tends to zero when $\kappa$ grows. 

Due to the close relation between the preconditioner and the original problem, we can estimate the pseudospectrum using a problem specific technique. 

\begin{lemma} 
\label{lemma:sl_inclusion}
There exists a positive constant $C>0$ such that 
\begin{equation*}
\left\{z \in \mathbb{C}\; \Big| \; C \left(  \frac{1}{ |z|^2 - \Re{z} }  + \frac{1}{|1-z|} \right) < \frac{1}{\epsilon} \; \right\} \setminus \overline{B \left( \frac{1}{2}, \frac{1}{2} \right)} \subset \Lambda_\epsilon^c
\end{equation*}
\end{lemma}
\begin{proof} There holds that $A = A^T$ and $B = B^T$. Using the identity $|C| = | C^*|$  for any $C \in \mathbb{C}^{N \times N}$ , it follows that 
\begin{equation*}
\sup_{ \vec{x}_u \in \mathbb{C}^n } \frac{  \bm{|} ( zI -  A B^{-1} )^{-1} \vec{x}_u\bm{|} }{ \bm{|} \vec{x}_u \bm{|} } = \sup_{ \vec{x}_u \in \mathbb{C}^n } \frac{  | ( z B -A)^{-1} B \vec{x}_u | }{ | \vec{x}_u | } 
\end{equation*}
Let $\vec{x}_v \in \mathbb{C}^N$ be such that  
\begin{equation*}
(A - zB) \vec{x}_v =  B \vec{x}_u.
\end{equation*}
As in Section 3, we establish the stability estimate $| \vec{x}_v| \leq f(z) |\vec{x}_u|$. When $f(z)$ is finite, this estimate yields the desired bound. Testing with any $\vec{x}_w \in \mathbb{R}^N$  gives
\begin{equation*}
(1-z) a(v,w ) - \mathrm{i}\sigma z (v,w) =  \vec{x}_w^*B \vec{x}_u.
\end{equation*}
Assuming that $z \neq 1$ and dividing by $1-z$ yields
\begin{equation*}
a(v,w) - \frac{ \mathrm{i}\sigma z}{1-z} (v, w) =   a(\frac{u}{1-z} ,w) + \frac{\mathrm{i}{\sigma}}{1-z} ( u,w).
\end{equation*}
By adding an subtracting a suitable term, the above can be written as 
\begin{equation*}
a(v - (1-z)^{-1}u,w) - \frac{ \mathrm{i}\sigma z}{1-z} (v - (1-z)^{-1}u, w) = \frac{ \mathrm{i}\sigma }{(1-z)^2} (u,w)
\end{equation*}
Choosing $w = v - (1-z)^{-1}u$,  using the identity $\frac{z}{1-z} = \frac{z - |z|^2}{|1-z|^2}$ and taking  imaginary part gives
\begin{equation*}
\kappa \|v - (1-z)^{-1}u  \|_{0,\partial \Omega}^2 + \sigma \frac{ |z|^2 - \Re{z} }{|1-z|^2} \| v - (1-z)^{-1}u \|_0^2 = \Im \frac{ \mathrm{i}\sigma }{(1-z)^2} (u,v - (1-z)^{-1}u ) 
\end{equation*}
When $z \neq 1$ and $\Re{z} - |z|^2 > 0$, this is  
 \begin{equation*}
\left| \frac{1}{2}-z \right| \geq \frac{1}{2}, 
 \end{equation*}
the coefficient of the $L^2(\Omega)$ - term is positive and we obtain the estimate 
\begin{equation*}
\| v - (1-z)^{-1} u \|_0 \leq \frac{1}{ |z|^2 - \Re{z} }\|u \|_0 .
 \end{equation*}
Using the norm equivalence given in equation (\ref{eq:L2_2_EUCLIDIAN}) yields 
\begin{equation*}
| \vec{x}_v - (1-z)^{-1} \vec{x}_u | \leq C \frac{1}{ |z|^2 - \Re{z} } | \vec{x}_u |.
\end{equation*}
The stability estimate follows from  the above equation and triangle inequality as
\begin{equation*}
| \vec{x}_v | \leq  | \vec{x}_v - (1-z)^{-1} \vec{x}_u | + \frac{ |\vec{x}_u|}{|1-z|} \leq  C \left(  \frac{1}{ |z|^2 - \Re{z} } | + \frac{1}{|1-z|} \right) | \vec{x}_u |. \end{equation*}
\end{proof}.

%

To obtain an overview of the derived bounds we have computed the pseudospectrum for $\kappa = 16 \pi$ and $\sigma = 0.5 \kappa, 0.5 \kappa^2$ using the level three mesh. The results are presented in Figure \ref{fig:sl_1}. Based on these results, analysis given in this Section seems to capture the behavior of the pseudospectrum rather well. In both cases, when $\epsilon$ is sufficiently small, pseudospectrum is located inside the disc $B(\frac{1}{2}, \frac{1}{2} )$ ,  as predicted by Lemma \ref{lemma:sl_inclusion}.  When the loss term is small, the pseudospectrum has a rather small diameter and is located close to ${1}$. For large values of $\sigma$, the set moves closer to the origin. These results are in good agreement with the exclusion given in equation (\ref{eq:sl_exclusion}).

The GMRES convergence bound gives usable information only if the origin is outside the approximate pseudospectrum. In the current case, this requirement limits the value of $\epsilon$ and thus determines the GMRES convergence rate. We have studied the pseudospectrum close to the origin in more detail by using a bisection search to find $x \in \mathbb{R}$ closest to the origin such that $|(x I - AB^{-1})^{-1}| = 2 \cdot 10^{-2}$ for different $\kappa$ between $4 \pi$ and $64 \pi$ for $\sigma=0.5 \kappa$ and $\sigma = 0.5 \kappa^2$. The results are visualized in Figure \ref{fig:sl_2}. These results indicate, that the exclusion given in equation (\ref{eq:sl_exclusion}) corresponds well with the real behavior of the set.

\begin{figure}[ht]
\includegraphics[scale=0.45]{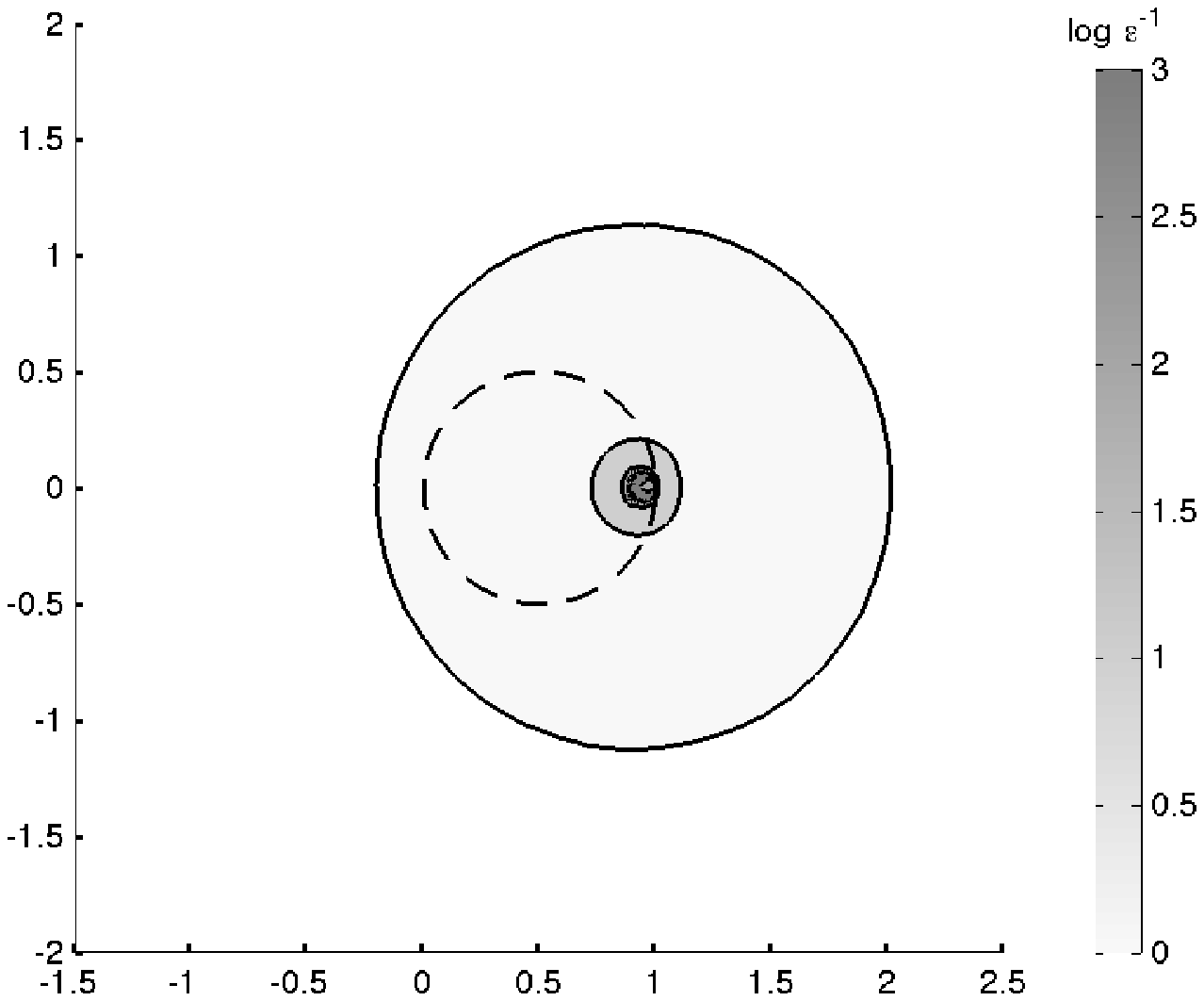}
\includegraphics[scale=0.45]{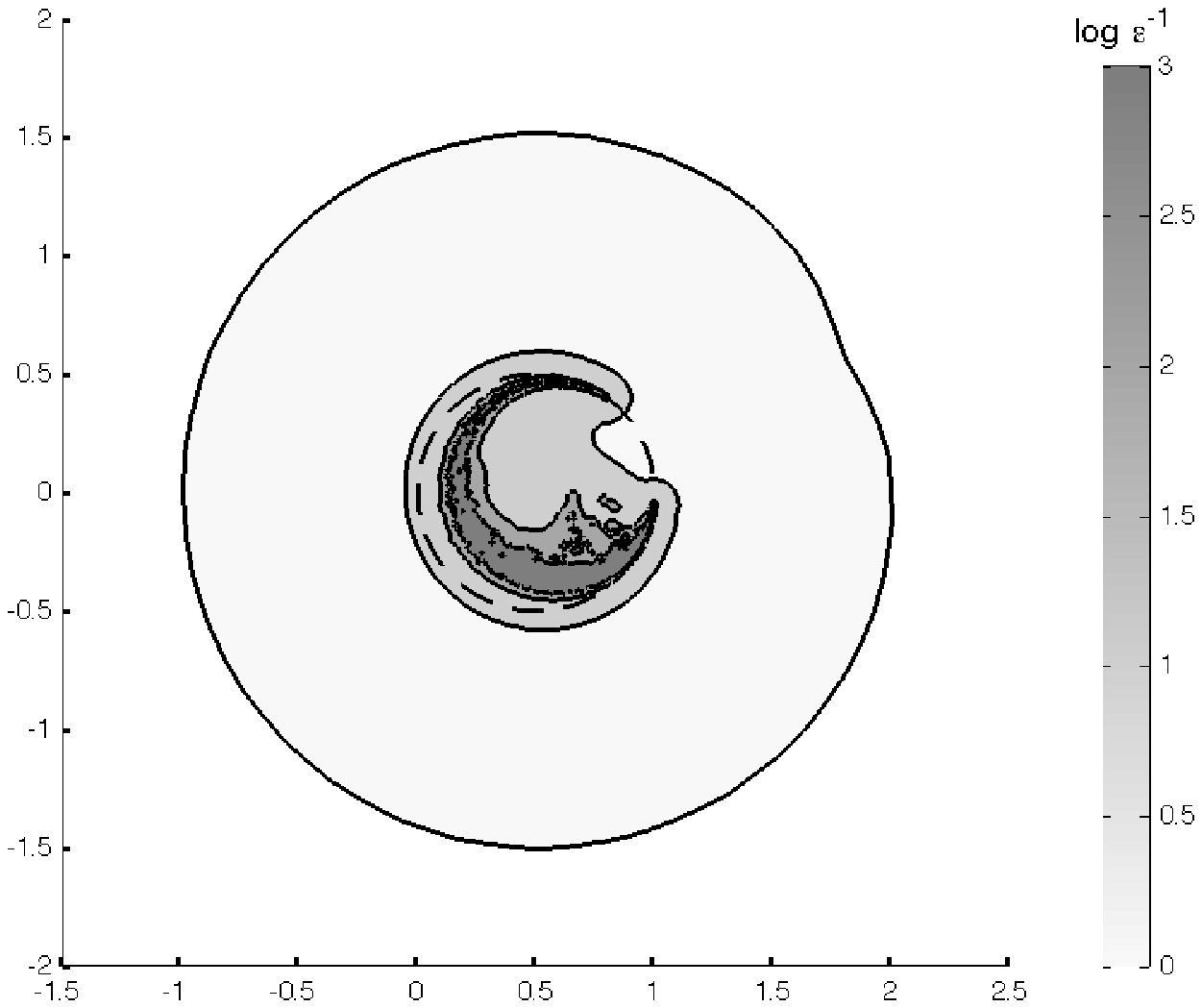}
\caption{Pseudospectrum for Example 4.3 with $\epsilon = 1,10,100,1000$. The parameter $\kappa = 16 \pi$ and level three mesh was used. On left the loss term is chose as $\sigma = 0.5 \kappa$ and on right as $\sigma = 0.5 \kappa^2$. The circle $B(\frac{1}{2},\frac{1}{2})$ is visualized with a dashed line.}
\label{fig:sl_1}
\end{figure}

\begin{figure}[ht]
\begin{minipage}[t]{0.45\linewidth}
\includegraphics[scale=0.5]{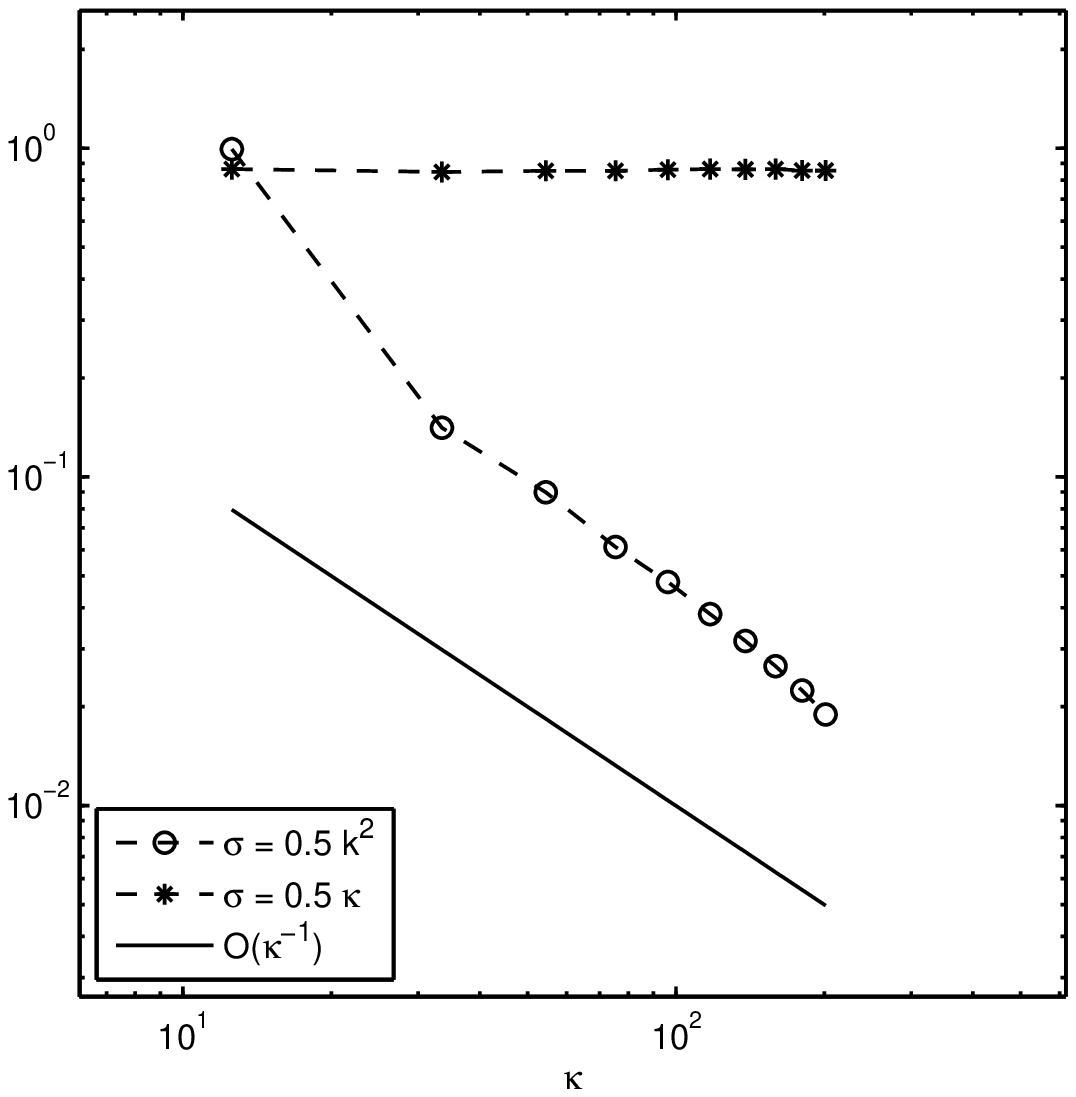}
\caption{ The point $x \in \mathbb{R}$ closest to the origin such that $|(x I - A )^{-1}| = 2 \cdot 10^{-2}$ for Example 4.3. The dependency is as predicted by the exclusion. Mesh level five was used in this computation. }
\label{fig:sl_2}
\end{minipage}
\hfill
\begin{minipage}[t]{0.45\textwidth}
\includegraphics[scale=0.5]{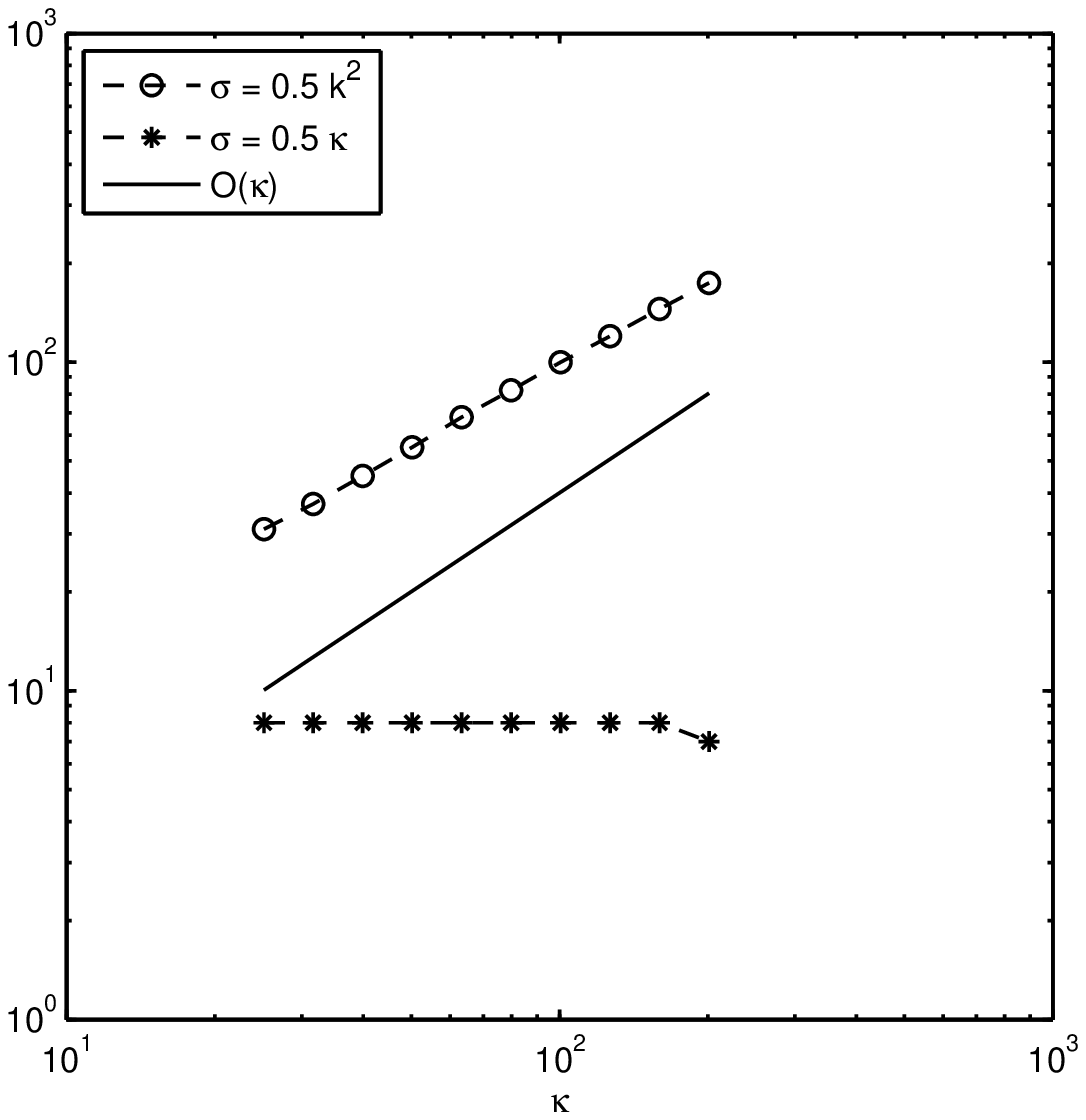}
\caption{The number of GMRES iterations required to solve the problem for different values of $\kappa$. The stopping criteria was set as $tol = 10^{-6}$. Level seven mesh was used in the computation. }
\label{fig:sl_3}
\end{minipage}
\end{figure}

A rigorous derivation of convergence estimate based on bratwurst shaped domains would require us to relate the parameters of these domains to $\tilde{\Lambda}_\epsilon$, which is out of the scope of this paper. Our computations indicate that the pseudospectrum for sufficiently large $\epsilon$ can be contained inside a circle, hence we will instead use the bound for circles to derive an approximate convergence rate. Based on the numerical and theoretical results, it seems to be reasonable to choose 
\begin{equation*}
\tilde{\Lambda}_\epsilon =  B(1, 1 - \frac{\kappa}{\kappa+\sigma}) \oplus B(0,\epsilon).
\end{equation*}
When $\sigma = 0.5 \kappa^2$, there holds that  $\tilde{\Lambda}_\epsilon \subset B(1, 1 - 0.5 \kappa^{-1} ) \oplus B(0,\epsilon)$. To exclude the origin, we choose $\epsilon = 0.25 \kappa^{-1}$. Using equation (\ref{eq:GMRES_B2}) and polynomial minimization over circles \cite{Gr:1997}, this leads to the estimate
\begin{equation*}
\frac{ |\vec{r}_i |}{| \vec{r}_0|}\leq 4 \kappa \left( \frac{1}{1 + 0.25 \kappa^{-1} }  \right)^{i}.
\end{equation*}
This, is the required number of iterations $N$ to reach tolerance $tol$ is
\begin{equation}
\label{eq:N_sl}
N \approx  - 4 \kappa \log{tol}+ 4 \kappa \log{  \kappa}.
\end{equation} 
So, asymptotically, the dominating term is $\kappa \log{ \kappa}$. We cannot observe this effect in our numerical examples as it would require us to use extremely large values of $\kappa$. For instance, when $\epsilon = 10^{-6}$, $\kappa$ would need to be of the order $10^{6}$, before it has an impact on the required number of GMRES iterations. This means, that the non-normality is not practically relevant in this case. 

Estimate (\ref{eq:N_sl}) rises the question, how should the stopping tolerance $tol$ be chosen. Using the tools derived in this paper, the size of relative residual can be related to $\kappa$-dependent norm. As we have studied right preconditioning, the solution obtained from GMRES $\vec{x}_i = B^{-1}\tilde{\vec{x}}_i$ satisfies $\vec{r}_i = A \vec{x}_i - \vec{b} = AB^{-1}\tilde{\vec{x}}_i-\vec{b}$. Hence, we can derive the estimate for the system without a preconditioner. The derived result holds for all left preconditioned systems. 
\begin{lemma} Consider the problem $A \vec{x}_u = \vec{b}$, in which $A \in \mathbb{C}^{N\times N}$ and $\vec{b} \in \mathbb{C}^N$ are related to the finite element discretization of problem (\ref{eq:weak_abs}). Let $\vec{x}_{\tilde{u} }$ be such that $|A \vec{x}_{\tilde{u} }-\vec{b}| \leq tol \; | \vec{b} |$. In addition, let Assumption \ref{ass3} hold. Then there exists a constant $C>0$ independent of $tol, u, \tilde{u}, \kappa$ and $h$ such that
\begin{equation*}
\| u - \tilde{u} \|_\kappa  \leq C tol \; \left(  \|f \|_0  +  h^{-1/2}\| g \|_{0,\partial \Omega} \right).
\end{equation*}
\end{lemma}
\begin{proof} Denote $\vec{r} = A \vec{x}_{\tilde{u} }-\vec{b}$. There holds that $A \vec{x}_{\tilde{u} }-\vec{b} = A( \vec{x}_{\tilde{u}} - \vec{x}_u )$, hence,  error $\vec{e} = \vec{x}_{\tilde{u}} - \vec{x}_u$ is a solution to the equation, 
\begin{equation*}
A \vec{e} = \vec{r}. 
\end{equation*} 
Let the space $W=L^2(\Omega)$. Using the same construction as in the proof of Corollary \ref{co:stab}, we define $q \in V$ such that $(q,v)_W = \vec{r}^*\vec{x_v} \; \forall v \in V$ and the linear functional $L(w) = (q,w)_W$. Using standard tools and the norm equivalence (\ref{eq:L2_2_EUCLIDIAN}) gives $\|L \|_{W^\prime}  \leq C h^{-d/2} | \vec{r}  |$. The stability estimate given in equation (\ref{eq:abs_stab}) leads to
\begin{equation*}
\|u - \tilde{u}\|_\kappa\leq C h^{-d/2} | \vec{r} |.
\end{equation*}
Now, this can be written as 
\begin{equation*}
\| u - \tilde{u} \|_\kappa \leq C h^{-d/2} tol \; | \vec{b} |.
\end{equation*}
As there holds that 
\begin{equation*}
| \vec{b} | = \max_{x_v \in \mathbb{R}^N} \frac{ \vec{b}^* \vec{x}_v}{ | \vec{x}_v| } =  \frac{ (f,v) + \left( g,v \right)_{\partial\Omega} }{ | \vec{x}_v|},
\end{equation*}
Cauchy-Schwarz inequality and norm equivivalence (\ref{eq:L2_2_EUCLIDIAN}) gives 
\begin{equation*} | \vec{b}| \leq C h^{d/2} \left( \|f \|_0 + h^{-1/2} \|g \|_0 \right) \end{equation*}
\end{proof}

One should note that identical techniques that were used to prove the above Lemma can be used to derive a relation between the $V$ - norm and $tol$ for any finite dimensional variational problem satisfying Assumption \ref{ass1}.    

We conclude by solving the shifted-Laplace preconditioned problem for right-hand side 
\begin{equation*}
f = \exp{ \left(-10^3 ( (x-0.5)^2 + (y+0.5)^2) \right)}
\end{equation*}
and different values of $\kappa$. The loss term for the preconditioner was chosen as $\sigma = 0.5 \kappa^2$ and $0.5 \kappa$ and the level five mesh was used in the computations.  The number of GMRES iterations is plotted in Figure \ref{fig:sl_3}. In this case, we observe a linear relationship between $\kappa$ and the number of iterations for $\sigma = 0.5 \kappa^2$. The number of iterations stays constant for $\sigma = 0.5 \kappa$. These results are in good agreement with the estimate (\ref{eq:N_sl}).

\section{Conclusions} The main result of the paper is the derivation of exclusion region for pseudospectral set near the origin, Theorem \ref{th:exclusion}. The derivation was made under Assumption \ref{ass1}, stability of the weak problem. All analysis was done a priori, without constructing any matrices. Theorem \ref{th:exclusion} was applied in all three tests, and the derived results were in good agreement with the true behavior of the pseudospectral set. In addition, an inclusion region was derived using the connection between FOV and the pseudospectrum. Boundedness estimates for FOV were derived based on the properties of the weak problem. All given analysis is applicable to a wide range of different problems. 
  
As demonstrated by the examples, the proposed inclusion and exclusion regions led to a worst case convergence estimate for the GMRES method. However, the effect of this overestimation was significant only for extreme parameter values. As illustrated by the first example, more refined convergence estimate would require knowledge from behavior of pseudospectrum for $\epsilon \rightarrow 0$. Such analysis is one direction for continuing this work. 

The aim of the paper was to investigate, if pseudospectrum based convergence estimate can be used for relating properties of weak form to convergence of GMRES. This was proven to be possible. As in Example 4.3, one needs to establish stability and boundedness of the preconditioned problem. The application of the derived theory will lead to inclusion and exclusion regions for pseudospectrum. Second possible direction for future work is to study different preconditioners and problems using the derived tools. Natural extension would be to investigate convergence of GMRES for time-harmonic Maxwell's equations.

\bibliographystyle{plain}      
\bibliography{pseudo_bib}

\end{document}